\documentclass[a4paper,10pt]{article}
\usepackage[utf8]{inputenc}

\usepackage{wrapfig}
\usepackage{hyperref}
\usepackage{amsmath} 
\usepackage{amsthm} 
\usepackage{amssymb} 
\usepackage{enumerate} 
\usepackage{esint} 
\usepackage{pgf,tikz} 
 \usepackage{yfonts} 
 \usepackage{mathrsfs} 
 \usepackage{mathabx} 
 \usepackage{graphicx}
\usepackage{caption}
\usepackage{subfigure}
\usepackage{mathtools} 

\usetikzlibrary{arrows,chains,matrix,positioning,scopes} 
\makeatletter
\tikzset{join/.code=\tikzset{after node path={%
\ifx\tikzchainprevious\pgfutil@empty\else(\tikzchainprevious)%
edge[every join]#1(\tikzchaincurrent)\fi}}}
\makeatother
\tikzset{>=stealth',every on chain/.append style={join},
         every join/.style={->}}
\tikzstyle{labeled}=[execute at begin node=$\scriptstyle,
   execute at end node=$]

\definecolor{ffffff}{rgb}{1.0,1.0,1.0}
\definecolor{qqqqff}{rgb}{0.0,0.0,1.0}
\definecolor{ffqqqq}{rgb}{1.0,0.0,0.0}
\definecolor{zzzzqq}{rgb}{0.6,0.6,0.0}
\definecolor{uququq}{rgb}{0,0,0}
\definecolor{qqttcc}{rgb}{0.,0.2,0.8}
\definecolor{ffqqqq}{rgb}{1.,0.,0.}
\definecolor{cqcqcq}{rgb}{0.7529411764705882,0.7529411764705882,0.7529411764705882}
\definecolor{ffffff}{rgb}{1.0,1.0,1.0}
\definecolor{qqqqff}{rgb}{0.0,0.0,1.0}
\definecolor{ffqqqq}{rgb}{1.0,0.0,0.0}
\definecolor{zzzzqq}{rgb}{0.6,0.6,0.0}
\definecolor{marronet}{rgb}{0.6,0.2,0}
\definecolor{negre}{rgb}{0,0,0}
\definecolor{vermell}{rgb}{0.8,0.05,0.05}
\definecolor{blau}{rgb}{0.2,0.1,1}
\definecolor{blauclar}{rgb}{0.,0.,1.}
\definecolor{grisfosc}{rgb}{0.25098039215686274,0.25098039215686274,0.25098039215686274}
\definecolor{verd}{rgb}{0.1,0.6,0.1}
\definecolor{taronja}{rgb}{0.9,0.6,0.05}
\definecolor{vermellclar}{rgb}{1.,0.,0.}
\definecolor{verdet}{rgb}{0,0.8,0.1}
\definecolor{blauverd}{rgb}{0,0.4,0.2}
\definecolor{grisclar}{rgb}{0.6274509803921569,0.6274509803921569,0.6274509803921569}

\newcommand{\C}{{\mathbb C}}       
\newcommand{\R}{{\mathbb R}}       
\newcommand{\N}{{\mathbb N}}       
\newcommand{\DDD}{{\mathbb D}}

\newcommand{\rf}[1]{{(\ref{#1})}}
\newcommand{\supp}{{\rm supp}}

\newcommand{\ve}{{\varepsilon}}

\newcommand{\norm}[1]{{\left\| {#1} \right\|}}

\newtheorem{theorem}{Theorem}
\newtheorem*{theorem*}{Theorem}
\newtheorem{lemma}[theorem]{Lemma}

\newtheorem{corollary}[theorem]{Corollary}
\newtheorem*{corollary*}{Corollary}
\newtheorem{proposition}[theorem]{Proposition}
\newtheorem{definition}[theorem]{Definition}

\newtheorem{remark}[theorem]{Remark}

\numberwithin{subsection}{section}
\numberwithin{theorem}{section}
\numberwithin{equation}{section}
\numberwithin{figure}{section}

\usepackage[affil-it]{authblk}

\title{Sharp bounds for composition with quasiconformal mappings in Sobolev spaces}

\author{Marcos Oliva and Mart\'i Prats
\thanks{Departamento de Ma\-te\-m\'a\-ti\-cas, Universidad Aut\'onoma de Madrid-ICMAT, Spain,  \texttt{marcos.delaoliva@uam.es}, \texttt{marti.prats@uam.es}.}}

\begin{document}
\maketitle
\bibliographystyle{alpha}

\begin{abstract} 
Let $\phi$ be a quasiconformal mapping, and let $T_\phi$ be the composition operator which maps $f$ to $f\circ\phi$. Since $\phi$ may not be bi-Lipschitz, the composition operator need not map Sobolev spaces to themselves.   The study begins with the behavior of $T_\phi$ on $L^p$ and  $W^{1,p}$ for $1<p<\infty$. This cases are well understood but alternative proofs of some known results are provided. Using interpolation techniques it is seen that compactly supported Bessel potential  functions in $H^{s,p}$ are sent to $H^{s,q}$ whenever $0<s<1$ for appropriate values of $q$.
The techniques used lead to sharp results and they can be applied to Besov  spaces as well.
\end{abstract}

\renewcommand{\abstractname}{}
\begin{abstract}

\end{abstract}

\section{Introduction}
Given a quasiconformal homeomorphism $\phi:\Omega\to\Omega'$ between domains in $\R^n$, we consider the composition operator $T_\phi$ which maps every measurable function $f: \Omega' \to \R$ to $f\circ\phi$. It is well known that $\phi$ lies in a certain Sobolev space $W^{1,p}_{loc}$ with $p>n$,  that is, the space of locally $p$-integrable functions with locally $p-$integrable derivatives, and in some H\"older class $C^s$ with $0<s<1$, i.e. for any  $K\subset\Omega$ compact $\phi$ is bounded and continuous in $K$ with $|\phi(x)-\phi(y)|\leq C_K|x-y|^s$ for $x,y\in K$. The composition operator $T_\phi$ is a self-map of $W^{1,n}$. However, since $\phi$ may not be bi-Lipschitz, the composition operator does not necessarily map other Sobolev spaces to themselves.  

A characterization of homeomorphisms which give rise to bounded composition operators is given in \cite{Ukhlov} and \cite{Kleprlik}.

Over the last decade, the study on the stability of the planar Calder\'on inverse conductivity problem raised questions on the range of $T_\phi$. The ground-breaking work of Astala and P\"aiv\"arinta \cite{AstalaPaivarinta} showing uniqueness of the solution to the problem was adapted by Barcel\'o, Faraco and Ruiz  in \cite{BarceloFaracoRuiz} to provide stability in Lipschitz domains with only H\"older a priori conditions on the conductivities. Some years later, Clop, Faraco and Ruiz weakened the a priori assumption on the conductivities to just a fractional Sobolev condition in \cite{ClopFaracoRuiz}, allowing the method to be applied to non-continuous conductivities, and later on the regularity assumptions on the boundary of the domain where severely reduced in \cite{FaracoRogers}. A deeper knowledge of the behavior of the composition operator may lead to better numerical methods for the electric impedance tomography (see \cite{AstalaMuellerPaivarintaSiltanen}, for instance).

We study quasiconformal mappings $\phi$ whose Jacobian determinant $J_\phi$ satisfies the estimates
\begin{equation}\label{eqAKBKRn}
\left(\int_U J_\phi(x)^{a} \, dx\right)^\frac1a \leq C_a, \mbox{\quad\quad and \quad\quad} \left(\int_U J_\phi(x)^{-b} \, dx\right)^\frac1b \leq C_b
\end{equation}
 for values $a>1$ and $b>0$, where  $U\subset \R^n$ is a certain domain (open and connected set). The existence of these values for any domain $U$ compactly contained in $\Omega$ can be derived from quasiconformality itself, but we may have better exponents for particular mappings, and this will imply better behavior of the composition operator. 

In \cite[Theorems 1.1 and 1.2]{HenclKoskela}, it is shown that, under the first condition in \rf{eqAKBKRn} the composition operator $T_\phi$ sends compactly supported $W^{1,p}$ functions to $W^{1,q}$ for certain couples $n\leq q<p$, that is, some integrability of the function and its gradient is lost. The main result of the present paper shows that this loss of integrability is common to all the Bessel potential spaces $H^{s,p}$ with $0\leq s \leq 1$ and $1<p<\infty$ and Besov spaces $B^s_{p,r}$ with $0< s < 1$, $1<p<\infty$ and $0<r\leq\infty$ (see Section \ref{secFractional} for the definitions). Note that $H^{0,p}=L^p$ and $H^{1,p}=W^{1,p}$.

\begin{theorem}\label{theoCompositionQCSobolevFractionalMain}
 Let $n\ge 2$, $0\leq s \leq 1$ and $1<p < \infty$. Given a quasiconformal mapping $\phi:\Omega \to \Omega'$ between two domains in $\R^{n}$, a ball $B$ with $2B\subset \Omega$, a ball $B'$ with $\phi(2B)\subset B'$ and positive real numbers $a$, $b$, $C_a$ and $C_b$ satisfying \rf{eqAKBKRn} for $U=2B$, let $q$ be defined by  $\frac1q = \frac1p + \frac1{c}\left|\frac{s}{n}-\frac1p\right|$, where we take $c=a$ if $sp \geq n$ and $c=b$ if $sp < n$. 
 If $q >1$,  then there exists a constant $C$ such that
\begin{equation}\label{eqTBoundedFractionalGeneral}
\norm{T_\phi f}_{ {H}^{s,q}(B)} \leq C \norm{f}_{{H}^{s,p}(B')}
\end{equation} 
and, if in addition $s\notin \{0,1\}$, then
\begin{equation}\label{eqTBoundedFractionalGeneralOther}
\norm{T_\phi f}_{B^s_{q,r}(B)} \leq C \norm{f}_{B^s_{p,r}(B')}
\end{equation} 
for every locally integrable function $f$ and every $0< r\leq \infty$, with constants not depending on $\phi$. If $s\in\{0,1\}$ and $p=\infty$, then \rf{eqTBoundedFractionalGeneral} holds as well.
\end{theorem}

Previous results (\cite[Proposition 4.2]{ClopFaracoRuiz} and \cite[Theorem 1.2]{HenclKoskelaComposition}) show that $T_\phi$ sends compactly supported $H^{s,q}\cap L^\infty$ functions (with $0<s<1$) to $H^{\beta,q}$ functions for certain $\beta<s$, that is, with a loss on the smoothness parameter. More precisely, the statement of the latter theorem settles that question for mappings between diagonal Besov spaces $B^s_{q,q}\to B^\beta_{q,q}$, and it establishes the supremum of the admissible values of $\beta$ as $\frac{bs}{b+1-\frac{sq}{n}}=\frac{\frac sq}{\frac 1q + \frac 1b \left(\frac 1q -\frac sn\right)}$. 

These results can be recovered from Theorem \ref{theoCompositionQCSobolevFractionalMain}. Indeed, by \cite[Theorem 2.2.5]{RunstSickel}, for $0<s<\infty$, $0<q<\infty$, $0<r,\ell \leq\infty$ and $0<\Theta<1$, 
$$\norm{f}_{F^{\Theta s}_{\frac q\Theta, r}}\leq C_{s,q,r,\Theta} \norm{f}_{F^s_{q,\ell}}\norm{f}_{L^\infty} \mbox{\quad and \quad} \norm{f}_{B^{\Theta s}_{\frac q\Theta, \frac r\Theta}}\leq C_{s,q,r,\Theta} \norm{f}_{B^s_{q,r}}\norm{f}_{L^\infty}.$$
For $sq<n$, taking $\beta:=\frac{\frac sq}{\frac 1q + \frac 1b \left(\frac 1q -\frac sn\right)}$ and $p$ so that $\frac1q = \frac1p + \frac1{b}\left(\frac1p-\frac{\beta}{n}\right)$ one can check that $\frac{\beta}s=\frac{1/p}{1/q}$. Thus, chosing $\Theta = \frac qp$, $r:=q$ and given a function $ f\in B^s_{q,q} \cap L^\infty$ with compact support, we have that $\norm{f}_{B^{\beta}_{p, p}}\leq C \norm{f}_{B^s_{q,q}}\norm{f}_{L^\infty}.$
Using Theorem \ref{theoCompositionQCSobolevFractionalMain}, it follows that 
\begin{equation*}
\norm{T_\phi f}_{B^\beta_{q,p} (B)} \leq C_{s,p,q} \norm{f}_{B^\beta_{p,p} (B)}\leq C \norm{f}_{B^s_{q,q}}\norm{f}_{L^\infty},
\end{equation*} 
 recovering  \cite[Theorem 1.2]{HenclKoskelaComposition} by elementary embeddings and obtaining the end-point with a change on the secondary integrability index. In Bessel potential spaces the estimate reads as
\begin{equation*}
\norm{T_\phi f}_{H^{\beta,q} (B)} \leq C_{s,p,q} \norm{f}_{H^{\beta,p} (B)}\leq C \norm{f}_{H^{s,q}}\norm{f}_{L^\infty}.
\end{equation*}

However, the result in  \cite[Theorem 1.1]{HenclKoskelaComposition} can only be partially recovered. Indeed, condition \rf{eqAKBKRn} implies that $\phi^{-1}\in W^{1,n(b+1)}(U)$, arguing as in \rf{eqChangeInB} below. Thus,  $\phi^{-1}\in C^{\frac{b}{b+1}}(U)$. Therefore, $\phi$ satisfies that
\begin{equation}\label{eqConditionBalls}
|\phi(B)| \geq C |B|^{\alpha}
\end{equation}
for every ball $B\subset U$, with $\alpha=\frac{b+1}{b}$. According to \cite[Theorem 1.1]{HenclKoskelaComposition} this implies that given a function $f\in B^s_{q,q}(\phi(B))$, the composition operator maps it to $f\circ\phi\in B^\beta_{q,q}(B)$, where $\beta=\frac nq - \alpha \left(\frac nq -s\right)$.
As before, one can check that this statement when $\alpha = \frac{b+1}{b}$ is a consequence of Theorem \ref{theoCompositionQCSobolevFractionalMain} using the embeddings
$$\norm{f}_{F^{\beta}_{p, r}}\leq C_{s,q,r,\ell} \norm{f}_{F^s_{q,\ell}} \mbox{\quad and \quad} \norm{f}_{B^{\beta}_{p, r}}\leq C_{s,q,r} \norm{f}_{B^s_{q,r}},$$
which hold because $\beta = s-\frac nq + \frac np$ (see \cite[Theorem 2.7.1]{TriebelTheory}). 
As the reader may note, there may be values of $\alpha< \frac{b+1}{b}$ for which \rf{eqConditionBalls} holds but with no counterpart in the spirit of \rf{eqAKBKRn}. The planar case illustrates how this result must be taken into account.

Indeed, given a $K$-quasiconformal mapping $\phi : \C \to \C$, we may choose 
$$- \mathbf{b_K} < - b < 0 < a < \mathbf{a_K}$$
where $\mathbf{b_K}:=\frac{1}{K-1}$ and let $\mathbf{a_K}:=\frac{K}{K-1}$. In \cite[Theorem 13.4.2]{AstalaIwaniecMartin} it is shown that $J_\phi^a$ and $J_\phi^{-b}$ are locally integrable, with
$$\fint_B J_\phi^a \leq \frac{C_K }{1- \frac{a}{\mathbf{a_K}}} \left(\frac{|\phi(B)|}{|B|}\right)^a,$$
and
$$\fint_B J_\phi^{-b} \leq \frac{C_K}{1- \frac{b}{\mathbf{b_K}}} \left(\frac{|B|}{|\phi(B)|}\right)^{b}.$$
 On the other hand, although for $b=\mathbf{b_K}$ this integral may blow up (see \cite[Theorem 13.2.3]{AstalaIwaniecMartin}, for instance), still we have the end-point H\"older regularity
$$|\phi(B)| \geq C |B|^{K}.$$
Combining Theorem \ref{theoCompositionQCSobolevFractionalMain} and \cite[Theorem 1.1]{HenclKoskelaComposition} we get the following corollary.

\begin{corollary}\label{coroCompositionQCSobolevPlanar}
 Let $K\geq 1$. Let  $0\leq s \leq 1$ and $1<p< \infty$. Given domains $\Omega,\Omega'\subset \C$, a  $K$-quasiconformal mapping $\phi:\Omega \to \Omega'$, a ball $B$ with $2B\subset \Omega$, a ball $B'$ with $\phi(2B)\subset B'$, if $1>\frac1q > \frac1p + \frac1{\mathbf{c_K}}\left|\frac{s}{2}-\frac1p\right|$ (where we take $\mathbf{c_K}=\mathbf{a_K}$ if $sp \geq 2$ and $\mathbf{c_K}=\mathbf{b_K}$ if $sp < 2$), there exists a constant $C$ such that
\begin{equation*}
\norm{T_\phi f}_{H^{s,q}(B)} \leq C\norm{f}_{H^{s,p}(B')}
\end{equation*} 
and, if $s\notin \{0,1\}$, then
\begin{equation*}
\norm{T_\phi f}_{B^s_{q,r}(B)} \leq C \norm{f}_{B^s_{p,r}(B')}.
\end{equation*} 
Moreover, whenever $sp<2$ and $\beta=s-(K-1)\left(\frac 2p - s\right)>0$, we have that
\begin{equation*}
\norm{T_\phi f}_{B^\beta_{p,p}(B)} \leq C\norm{f}_{B^s_{p,p}(B')}.
\end{equation*} 
\end{corollary}

The results in Theorem \ref{theoCompositionQCSobolevFractionalMain} and Corollary \ref{coroCompositionQCSobolevPlanar} are sharp concerning the loss of integrability. In the critical setting ($sp=n$), this fact is obvious since the composition is a self-map for these spaces. Both examples presented in \cite{HenclKoskelaComposition} illustrate the sharpness of the subcritical setting ($sp<n$), adapting the arguments above. In Section \ref{secExamples} we check that this extends to the Bessel-potential subcritical setting and we adapt one of these examples to the supercritical setting ($sp>n$). 

In conclusion, we have seen that better integrability properties of the derivatives of quasiconformal mappings imply less loss of integrability in the composition, while better H\"older regularity implies less loss of smoothness. In general, we expect the integrability properties to be an open condition for higher dimensions as it happens in dimension 2 (see \cite[Section 13.4.1]{AstalaIwaniecMartin}). In higher dimensions, it is conjectured that one can take $b<\mathbf{b_K}:=\frac{1}{K^{({1}/{n-1})}-1}$ in \rf{eqAKBKRn}. If this is the case, the H\"older condition of the inverse mapping of exponent $K^{-1/(n-1)}$, which coincides with the Sobolev embedding of the conjectured endpoint Sobolev space for the inverse mapping, is reached by quasisymmetry (see \cite{KoskelaNotes}), that is, a closed condition for the loss of smoothness. It remains to see if this extends to the supercritical case and what happens in the fractional spaces between both end-points. We believe that  these techniques can be applied to the finite distortion setting described in \cite{Kleprlik}.

The paper is structured as follows. In Section \ref{secClassical} we revisit the proofs of the loss of integrability in the classical Lebesgue and Sobolev spaces. In Section \ref{secFractional}, we derive the proof of Theorem \ref{theoCompositionQCSobolevFractionalMain} by a subtle interpolation of those classical results. Finally, in Section \ref{secExamples}, we check the sharpness of the main theorem.

\section{Classical spaces}\label{secClassical}
\subsection{Composition in Lebesgue spaces}\label{secLebesgue}
\begin{definition}
Let $ 1\leq K<\infty$ and let $\Omega, \Omega'$ be two domains in $\R^n$. We say that a homeomorphism $\phi:\Omega \to \Omega'$ is a $K$-quasiconformal mapping if  $\phi\in W^{1,1}_{loc}(\Omega)$, and the distributional Jacobian matrix $D\phi$ satisfies the distortion inequality
$$|D\phi(x)|^n \leq K |J_\phi(x)| \mbox{\quad\quad a.e. }x\in\Omega,$$
where  $J_\phi$ stands for the Jacobian determinant of $\phi$ and $|A|=\sup_{|h|=1}|A\cdot h|$ stands for the usual operator norm of a matrix $A$.
\end{definition}

 We recall some properties of $K$-quasiconformal mappings. 
\begin{theorem}[see {\cite[Remark 6.1, Theorem 6.3]{KoskelaNotes}}]\label{theoChangeRn}
Let $\phi:\Omega \to \Omega'$ be a $K$-quasiconformal mapping between planar domains. Then
\begin{enumerate}
\item $\phi\in W^{1,p}_{loc}$ for some $p>n$ depending on $K$.
\item Either $J_\phi>0$ almost everywhere or $J_\phi<0$ almost everywhere.
\item For $E\subset \Omega$ measurable, $|E|=0$ if and only if $|\phi(E)|=0$.
\item Given a ball $B\subset 2B\subset\Omega$ and a measurable set  $E\subset B$, there are constants $C$ and $\alpha$ depending on $n$ and $K$ such that 
$$\frac{|\phi(E)|}{|\phi(B)|}\leq C \left(\frac{|E|}{|B|}\right)^\alpha.$$
\item $\phi^{-1}$ is a $K^{n-1}$-quasiconformal mapping.
\item Moreover, given $f \in L^1(\Omega')$, we have that $f\circ\phi \, |J_{\phi}| \in L^1(\Omega)$ with
\begin{equation}\label{eqChainRule}
\int_\Omega f(\phi(z))|J_\phi(z)| \, dm(z) = \int_{\Omega'} f(w)\,dm(w).
\end{equation}
\end{enumerate}
\end{theorem}
From now on we assume $J_\phi >0$ almost everywhere.
\begin{corollary}\label{coroIntegrability}
 There exist  $a=a_K>1$ and $b=b_K>0$ such that for every $K$-quasiconformal mapping $\phi:\Omega \to \Omega'$ and every measurable set $U$ contained in a ball $B$ with $2B\subset \Omega$, there exist constants $C_a$, $C_b$ depending on $K$, $a$ (resp. $b$), $|B|$ and $|\phi(B)|$,  such that \rf{eqAKBKRn} holds.
\end{corollary}
\begin{proof}
Take $a_K=\frac{p}{n}$ from  the first property in the theorem above. 
$$\int_B J_\phi (x)^{a_K} \, dm(x) \leq  \norm{D \phi}_{L^{p}(B)} \leq C_{a}.$$
Take $b_K=a_{K^{n-1}}-1$ and the last item above:
\begin{equation}\label{eqChangeInB}
\int_B J_\phi (x)^{-b_K} \, dm(x)=\int_B J_\phi (x) J_\phi (x)^{-b_K-1} \, dm(x)=\int_{\phi(B)} J_{\phi^{-1}} (x)^{b_K+1} \, dm(x)\leq C_b .
\end{equation}
By the fifth property, the corollary follows.
\end{proof}

Next we see how does interact the composition operator with Lebesgue spaces. 

\begin{figure}[ht]
\centering
{
\begin{tikzpicture}[line cap=round,line join=round,>=triangle 45,x=4.5cm,y=1.5cm]
\draw[->,color=cqcqcq] (-0.1830204139167261,0.) -- (1.2429300847367257,0.);
\foreach \x in {,0.2,0.4,0.6,0.8,1.,1.2}
\draw[shift={(\x,0)},color=cqcqcq] (0pt,2pt) -- (0pt,-2pt);
\draw[color=cqcqcq] (0.,-0.8) -- (0.,0.8);
\foreach \y in {0}
\draw[shift={(0,\y)},color=cqcqcq] (2pt,0pt) -- (-2pt,0pt);
\clip(-0.1830204139167261,-0.8) rectangle (1.2429300847367257,0.8);
\fill[color=qqttcc,fill=qqttcc,fill opacity=0.5] (0.2618443656074706,0.) -- (0.2618443656074706,0.) -- (0.2618443656074706,0.) -- cycle;
\fill[color=qqttcc,fill=qqttcc,fill opacity=0.25] (0.6982516416199216,0.) -- (0.6982516416199216,0.) -- (0.6982516416199216,0.) -- cycle;
\draw [color=cqcqcq] (1.,-0.8892264630311122) -- (1.,3.6269454505314855);
\draw [color=cqcqcq,domain=-0.1830204139167261:1.2429300847367257] plot(\x,{(-1.-0.*\x)/-1.});
\draw (0.19,0.4) node[anchor=north west] {$L^p$};
\draw (0.63,0.4) node[anchor=north west] {$L^q$};
\draw (0.2618443656074706,0.)-- (0.2618443656074706,0.);
\draw (0.6982516416199216,0.)-- (0.6982516416199216,0.);
\draw (0.2618443656074706,0.)-- (0.2618443656074706,0.);
\draw (0.6982516416199216,0.)-- (0.6982516416199216,0.);
\draw [color=qqttcc] (0.2618443656074706,0.)-- (0.2618443656074706,0.);
\draw [color=qqttcc] (0.2618443656074706,0.)-- (0.2618443656074706,0.);
\draw [color=qqttcc] (0.2618443656074706,0.)-- (0.2618443656074706,0.);
\draw [color=qqttcc] (0.6982516416199216,0.)-- (0.6982516416199216,0.);
\draw [color=qqttcc] (0.6982516416199216,0.)-- (0.6982516416199216,0.);
\draw [color=qqttcc] (0.6982516416199216,0.)-- (0.6982516416199216,0.);
\draw (0.19,0.05) node[anchor=north west] {$\frac1p$};
\draw (0.67,0.05) node[anchor=north west] {$\frac1q$};
\draw [->,thick,color=ffqqqq] (0.2618443656074706,0.) -- (0.6982516416199216,0.);
\draw [color=ffqqqq] (0.43,0.05) node[anchor=north west] {$\frac1{bp}$};
\draw [->,thick,color=qqqqff] (0,0) -- (0.2618443656074706,0.);
\begin{scriptsize}
\draw [fill=ffqqqq] (0.2618443656074706,0.) circle (0.5pt);
\draw [fill=ffqqqq] (0.6982516416199216,0.) circle (0.5pt);
\draw [fill=ffqqqq] (0.2618443656074706,0.) circle (0.5pt);
\draw [fill=ffqqqq] (0.6982516416199216,0.) circle (0.5pt);
\draw [fill=ffqqqq] (0.2618443656074706,0.) circle (0.5pt);
\draw [fill=ffqqqq] (0.6982516416199216,0.) circle (0.5pt);
\draw [fill=ffqqqq] (0.,0.) circle (0.5pt);
\end{scriptsize}
\end{tikzpicture}
}
\caption{Action of the composition operator described in Lemma \ref{lemCompositionQCLebegueRn}. The gap (loss of integrability) is proportional to $\frac1p$.}\label{figActionInLebesgue}
\end{figure}
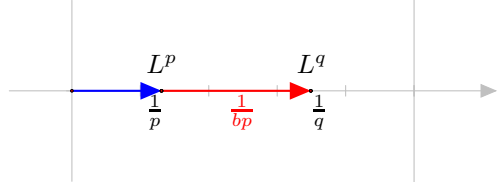

\begin{lemma}\label{lemCompositionQCLebegueRn}
 Let $n\geq 2$, $K\geq 1$ and $0<p \leq \infty$. Given a $K$-quasiconformal mapping $\phi:\Omega \to \Omega'$ between two domains in $\R^{n}$  and a function $f \in L^p(\Omega')$, let $b$ and $C_b$ satisfy \rf{eqAKBKRn} for $U=\Omega$ and let $q$ be defined as
$\frac1q = \frac1p + \frac{1}{b}\frac1p$. We have that
 $$\norm{f \circ \phi}_{L^q(\Omega)}\leq C_b^{\frac1p} \norm{f}_{L^p(\Omega')}.$$
\end{lemma}
\begin{proof}
The case $p=\infty$ being trivial, let us assume that $p<\infty$. By the H\"older inequality
\begin{align*}
 \norm{f \circ \phi}_{L^q (\Omega)}
 	& =  \norm{f \circ \phi  \, J_\phi ^\frac1p J_\phi^{-\frac1p} }_{L^q (\Omega)}\\
	& \leq  \left(\int_{\Omega}| f \circ \phi(x)|^p J_\phi(x) \, dm(x)\right)^\frac1p \left(\int_{\Omega} J_\phi(x)^{-b} \, dm(x)\right)^{\frac1b\frac1p}.
\end{align*}
Note that $\phi$ acts as a change of variables by Theorem \ref{theoChangeRn}, so
\begin{align*}
 \norm{f \circ \phi}_{L^q (\Omega)}
	& \leq C_b^{\frac1p} \norm{f}_{L^p(\phi(\Omega))} .
\end{align*}
\end{proof}

Let us write $B_1$ for the ball with radius $1$ centered at the origin. There exists a bump function $\varphi\in C^\infty_c(2B_1)$, such that $\chi_{B_1}\leq \varphi \leq \chi_{2B_1}$. For every ball $B$ we define $\varphi_B$ by precomposing $\varphi$ with an appropriate affine mapping, so that $\chi_{B}\leq \varphi_B \leq \chi_{2B}$.
\begin{definition}
Let $\phi:\Omega\to\Omega'$ be a $K$-quasiconformal mapping. Then, the composition operator $T_\phi:  L^{1+\frac1{b_K}}_{loc}(\Omega') \to L^1_{loc}(\Omega)$ is defined as
$$ T_\phi f := f\circ \phi.$$
Moreover, for every ball $B$ with $2B\subset \Omega$, we define
$$ T_\phi^B f := \varphi_B \cdot f\circ \phi.$$
\end{definition}

\begin{corollary}\label{coroCompositionLebesgueRn}
 Let $K\geq 1$ and $0<p \leq \infty$. Given a  $K$-quasiconformal mapping $\phi:\Omega \to \Omega'$ and $f \in L^p(\Omega')$. Let $B$ be a ball with $2B\subset \Omega$, and let $b$ and $C_b$ satisfy \rf{eqAKBKRn} for $U=2B$. If $\frac1q = \frac1p + \frac1{pb} $, then
 $$\norm{T_\phi^B}_{L^p \to L^q} \leq   C_b^{\frac1p}.$$
\end{corollary}

\subsection{Composition in classical Sobolev spaces}\label{secSobolev}
In this section we study the behavior of the composition operator in Sobolev spaces $W^{1,p}$ with $1\leq p \leq \infty$.
The following is partially contained in \cite[Theorems 1.1 and 1.2]{HenclKoskela}. In that case only the critical and supercritical cases are covered.

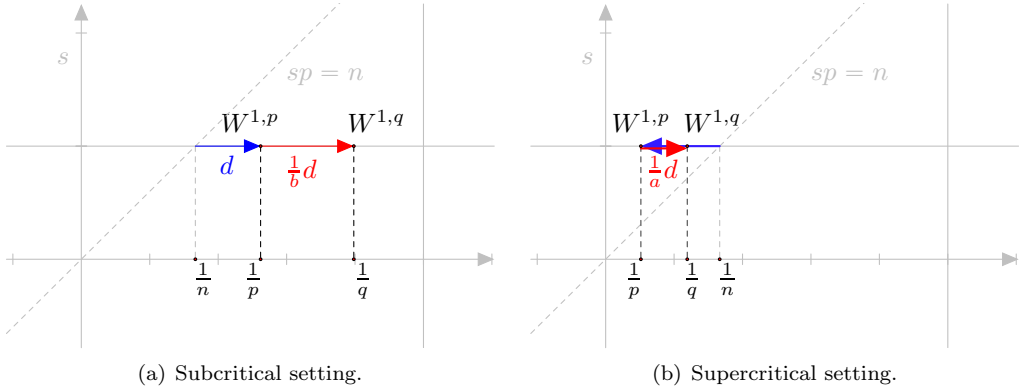
\begin{figure}[ht]
\centering
\subfigure[Subcritical setting.]
{
\begin{tikzpicture}[line cap=round,line join=round,>=triangle 45,x=4.5cm,y=1.5cm]
\draw[->,color=cqcqcq] (-0.21876847655706336,0.) -- (1.2071820220963887,0.);
\foreach \x in {-0.2,0.2,0.4,0.6,0.8,1.,1.2}
\draw[shift={(\x,0)},color=cqcqcq] (0pt,2pt) -- (0pt,-2pt);
\draw[->,color=cqcqcq] (0.,-0.7819822751100998) -- (0.,2.25660304931856);
\foreach \y in {,1.,2.}
\draw[shift={(0,\y)},color=cqcqcq] (2pt,0pt) -- (-2pt,0pt);
\clip(-0.21876847655706336,-0.7819822751100998) rectangle (1.2071820220963887,2.25660304931856);
\draw [color=cqcqcq] (1.,-0.7819822751100998) -- (1.,2.25660304931856);
\draw [color=cqcqcq,domain=-0.21876847655706336:1.2071820220963887] plot(\x,{(-1.-0.*\x)/-1.});
\draw (0.38,1.4) node[anchor=north west] {$W^{1,p}$};
\draw (0.75,1.4) node[anchor=north west] {$W^{1,q}$};
\draw [dash pattern=on 2pt off 2pt,color=cqcqcq,domain=-0.21876847655706336:1.2071820220963887] plot(\x,{(-0.--3.*\x)/1.});
\draw [dash pattern=on 2pt off 2pt] (0.5239968249699436,1.)-- (0.5239968249699436,0.);
\draw [dash pattern=on 2pt off 2pt] (0.7963732415936726,1.)-- (0.7963732415936726,0.);
\draw [dash pattern=on 2pt off 2pt,color=cqcqcq] (0.3333,1.)-- (0.3333,0.);
\draw [color=cqcqcq](0.5716609084903933,1.7680461932339522) node[anchor=north west] {$sp=n$};
\draw [color=cqcqcq](-0.09960826775593923,1.911038443795301) node[anchor=north west] {$s$};
\draw [->,color=ffqqqq] (0.5239968249699436,1.) -- (0.7963732415936726,1.);
\draw [color=ffqqqq](0.57,1.) node[anchor=north west] {$\frac1{b} d$};
\draw (0.3,0.05) node[anchor=north west] {$\frac1n$};
\draw (0.45,0.05) node[anchor=north west] {$\frac1p$};
\draw (0.77,0.05) node[anchor=north west] {$\frac1q$};
\draw [color=qqqqff](0.37703256744855723,1.) node[anchor=north west] {$d$};
\draw [->,color=qqqqff] (0.3333333333333333,1.) -- (0.5239968249699436,1.);
\begin{scriptsize}
\draw [fill=ffqqqq] (0.5239968249699436,1.) circle (0.5pt);
\draw [fill=ffqqqq] (0.7963732415936726,1.) circle (0.5pt);
\draw [fill=ffqqqq] (0.3333,0.) circle (0.5pt);
\draw [fill=ffqqqq] (0.5239968249699436,0.) circle (0.5pt);
\draw [fill=ffqqqq] (0.7963732415936726,0.) circle (0.5pt);
\end{scriptsize}
\end{tikzpicture}
}\,\,
\subfigure[Supercritical setting.]
{
\begin{tikzpicture}[line cap=round,line join=round,>=triangle 45,x=4.5cm,y=1.5cm]
\draw[->,color=cqcqcq] (-0.21876847655706336,0.) -- (1.2071820220963887,0.);
\foreach \x in {-0.2,0.2,0.4,0.6,0.8,1.,1.2}
\draw[shift={(\x,0)},color=cqcqcq] (0pt,2pt) -- (0pt,-2pt);
\draw[->,color=cqcqcq] (0.,-0.7819822751100998) -- (0.,2.25660304931856);
\foreach \y in {,1.,2.}
\draw[shift={(0,\y)},color=cqcqcq] (2pt,0pt) -- (-2pt,0pt);
\clip(-0.21876847655706336,-0.7819822751100998) rectangle (1.2071820220963887,2.25660304931856);
\draw [color=cqcqcq] (1.,-0.7819822751100998) -- (1.,2.25660304931856);
\draw [color=cqcqcq,domain=-0.21876847655706336:1.2071820220963887] plot(\x,{(-1.-0.*\x)/-1.});
\draw (-0.01,1.4) node[anchor=north west] {$W^{1,p}$};
\draw (0.2,1.4) node[anchor=north west] {$W^{1,q}$};
\draw [dash pattern=on 2pt off 2pt,color=cqcqcq,domain=-0.21876847655706336:1.2071820220963887] plot(\x,{(-0.--3.*\x)/1.});
\draw [dash pattern=on 2pt off 2pt,color=grisfosc](0.10296408720597178,1.)-- (0.10296408720597178,0.);
\draw [dash pattern=on 2pt off 2pt,color=grisfosc](0.23847540845736093,1.)-- (0.23847540845736093,0.);
\draw [dash pattern=on 2pt off 2pt,color=cqcqcq] (0.3333,1.)-- (0.3333,0.);
\draw  [color=cqcqcq](0.5716609084903933,1.7680461932339522) node[anchor=north west] {$sp=n$};
\draw [color=cqcqcq](-0.09960826775593923,1.911038443795301) node[anchor=north west] {$s$};
\draw [->,color=blau, thick] (0.3333333333333333,1.) -- (0.10296408720597178,1.);
\draw [->,color=ffqqqq, thick] (0.10296408720597178,0.98) -- (0.23847540845736093,0.98);
\draw [color=ffqqqq](0.09,1.) node[anchor=north west] {$\frac1a d$};
\draw (0.03,0.05) node[anchor=north west] {$\frac1p$};
\draw (0.2,0.05) node[anchor=north west] {$\frac1q$};
\draw (0.3,0.05) node[anchor=north west] {$\frac1n$};
\begin{scriptsize}
\draw [fill=ffqqqq] (0.10296408720597178,1.) circle (0.5pt);
\draw [fill=ffqqqq] (0.23847540845736093,1.) circle (0.5pt);
\draw [fill=ffqqqq] (0.3333,0.) circle (0.5pt);
\draw [fill=ffqqqq] (0.10296408720597178,0.) circle (0.5pt);
\draw [fill=ffqqqq] (0.23847540845736093,0.) circle (0.5pt);
\end{scriptsize}
\end{tikzpicture}
}
\caption{Action of the composition operator described in Theorem \ref{theoCompositionQCSobolev}. The gap (loss of integrability) is proportional to $d$, i.e. to the horizontal distance to the critical line (which is the homogeneity of the space seminorm under rescaling divided by $n$).}\label{figActionInSobolev}
\end{figure}

Next we present a particular case of \cite[Theorem 1.3]{Kleprlik}, which we prove for the sake of completeness and to keep control on the constants.
\begin{theorem}\label{theoCompositionQCSobolev}
 Let $n\ge 2$, $K\geq 1$ and $1<p \leq \infty$. Given a $K$-quasiconformal mapping $\phi:\Omega \to \Omega'$ between two domains in $\R^{n}$ and a function $f \in W^{1,p}(\Omega')$, let $a$, $b$, $C_a$ and $C_b$ be real numbers satisfying \rf{eqAKBKRn}  for $U=\Omega$ and let $q$ be defined as follows:
 \begin{itemize}
 \item If $n<p\leq \infty$, $\frac1q = \frac1p + \frac1{a}\left(\frac1n-\frac1p\right)$.
 \item If $p=n$, $\frac1q = \frac1p$.
 \item If $1+\frac{n-1}{nb +1} \leq p<n$, $\frac1q = \frac1p + \frac{1}{b}\left(\frac1p-\frac1n\right)$
 \end{itemize}
(see Figure \ref{figActionInSobolev}). We have that
 \begin{equation}\label{eqInequalityForGradientRn}
\norm{\nabla( f \circ \phi)}_{L^q(\Omega)}\leq K^\frac1n C_c^{\left|\frac1n-\frac1p\right|}  \norm{\nabla f}_{L^p(\Omega')},
 \end{equation}
 where $c$ stands for $a$ (resp. $b$ or $1$ with $C_1=1$) if we are in the supercritical (resp.  subcritical or critical) case.
\end{theorem}
Note that $1\geq \frac1q$ is granted in the supercritical and the critical cases, while in the subcritical it is equivalent to $p \geq 1+\frac{n-1}{nb +1}$.

\begin{proof}

By the chain rule (see \cite[Theorem 1.3]{Kleprlik} for instance) we have that
\begin{align*}
\norm{\nabla(f\circ \phi)}_{L^q(\Omega)}
	& \leq \left(\int_\Omega |\nabla f (\phi(z)))|^q |D^t\phi(z)|^q \right)^\frac1q
		\leq K^\frac1n \left(\int_\Omega |\nabla f (\phi(z)))|^q |J_\phi(z)|^\frac{q}n \right)^\frac1q.\\
\end{align*}
In case $p=q=n$, by \rf{eqChainRule} we have shown that
$$\norm{\nabla(f\circ \phi)}_{L^p(\Omega)}
	 \leq K^\frac1n \norm{\nabla f}_{L^p(\phi(\Omega))}.
$$

We need to study the supercritical and the subcritical cases. Let $1<p<\infty$ with $p\neq n$  (the case $p=\infty$ follows by an analogous reasoning). The H\"older inequality implies that
\begin{align*}
\norm{\nabla(f\circ \phi)}_{L^q(\Omega)}
	& \leq K^\frac1n \left(\int_\Omega |\nabla f (\phi(z))|^p |J_\phi(z)|\right)^\frac1p \left(\int_\Omega J_\phi(z)^{\left(\frac{q}n-\frac{q}p\right) \frac{p}{p-q}}\right)^\frac{p-q}{pq}\\
	& \leq K^\frac1n \left(\int_\Omega J_\phi(z)^{\frac{p-n}n \frac{q}{p-q}}\right)^\frac{p-q}{pq} \norm{\nabla f}_{L^p(\phi(\Omega))},
\end{align*}
and it remains to control the integral of the Jacobian.
For \rf{eqAKBKRn} to apply we need that $\frac{p-n}n \frac{q}{p-q} $ equals $a$ or $-b$ depending on the setting.

First we study the supercritical case, $p> n$. 
Since $p> q$, we have that $\frac{p-n}n \frac{q}{p-q} > 0$, and we only need to check that
$$\frac{p-n}{n} \frac{q}{p-q} = a,$$
that is, $\frac{p-q}{pq} = \frac{p-n}{n a p}$, which is equivalent to our first assumption
$$\frac1q =\frac1p + \frac{1}{a}\left(\frac1n-\frac1p\right).$$

Next we study the case $p< n$ and $p>q$. Here $\frac{p-n}n \frac{q}{p-q} < 0$, and we will check the other equality, that is $\frac{p-n}{n} \frac{q}{p-q} = -b$,
or, changing signs, 
$$\frac{n-p}{n} \frac{q}{p-q} = b,$$
which is equivalent to $\frac{p-q}{pq} = \frac{n-p}{nbp}$ and finally, to
$$\frac1q = \frac1p + \frac1b \left(\frac1p-\frac1n\right).$$
\end{proof}

\begin{corollary}\label{coroCompositionSobolevRn}
Let $K\geq 1$ and $1<p \leq \infty$. Given a  $K$-quasiconformal mapping $\phi:\Omega \to \Omega'$, let $B$ be a ball with $2B\subset \Omega$, let $B'$ a ball with $\phi(2B)\subset B'$, and let $b>0$ and $C_b$ satisfy \rf{eqAKBKRn} for $U=2B$. Let $q$ be defined as before.  
If $p>n$, then
$$\norm{T_\phi^B}_{{W}^{1,p} \to {W}^{1,q}} \leq C_{K,n,p,q,|B|,|B'|}\left(1+C_a^{\frac1n-\frac1p} \right).$$
If $p=n$, then
$$\norm{T_\phi^B}_{{W}^{1,n} \to {W}^{1,n}} \leq C_{K,n,b,C_b,|B|,|B'|}.$$
If $p<n$, then
$$\norm{T_\phi^B}_{{W}^{1,p} \to {W}^{1,q}} \leq C_{K,n,p,q,|B|,|B'|}  C_b^{\frac1p-\frac1n}.$$
\end{corollary}
\begin{proof}
We will assume that $B=B'=B_1$, which can be achieved by translation and rescaling. Let $f\in W^{1,p}$. First take $n> p > q$, and let $\frac1{q^*}=\frac1q-\frac1n$. By  Leibniz' rule and H\"older's inequality,
\begin{align}\label{eqBumpAndGoSubcritical}
\norm{\nabla(\varphi \cdot f\circ\phi)}_{L^q(2B_1)}
\nonumber	& \leq \norm{\nabla\varphi \cdot f\circ\phi}_{L^q(2B_1)} + \norm{\varphi \cdot \nabla(f\circ\phi)}_{L^q(2B_1)}\\
	& \leq \norm{\nabla\varphi}_{L^n} \norm{f\circ\phi}_{L^{q^*}(2B_1)} + \norm{\varphi}_{L^\infty}\norm{\nabla(f\circ\phi)}_{L^q(2B_1)}.
\end{align}
Since $\frac1{q^*}=\frac1q-\frac1n=\frac1{p^*}+\frac1b\frac1{p^*}$, Lemma \ref{lemCompositionQCLebegueRn} and the Gagliardo-Nirenberg-Sobolev inequality (see \cite[Theorem 5.6.1/2]{Evans}, for instance) give
 $$\norm{f \circ \phi}_{L^{q^*}(2B_1)}\leq C_b^{\frac1{p^*}} \norm{f}_{L^{p^*}(\phi(2B_1))} \leq  C_b^{\frac1p-\frac1n} \norm{f}_{L^{p^*}(B_1)}\leq C_{p,n} C_b^{\frac1p-\frac1n} \norm{f}_{W^{1,p}(B_1)}.$$
 Theorem \ref{theoCompositionQCSobolev} gives
 $$\norm{\nabla( f \circ \phi)}_{L^q(2B_1)}\leq K^\frac1n C_b^{\frac1p-\frac1n}  \norm{\nabla f}_{L^p(\phi(2B_1))}.$$
 Back to \rf{eqBumpAndGoSubcritical}, we have shown that
\begin{align*}
\norm{\nabla(\varphi_B \cdot f\circ\phi)}_{L^q(2B_1)}
	& \leq C_{K,n,p}  C_b^{\frac1p-\frac1n} \norm{f}_{{W}^{1,p}}.
\end{align*}
By the Poincar\'e inequality, 
\begin{align*}
\norm{\varphi \cdot f\circ\phi }_{{W}^{1,q}}
	& \leq C \norm{\nabla(\varphi \cdot f\circ\phi)}_{L^q(2B_1)}
		\leq C_{K,n,p}  C_b^{\frac1p-\frac1n} \norm{f}_{{W}^{1,p}}.
\end{align*}

The critical case, $p=q=n$, instead of \rf{eqBumpAndGoSubcritical} we use
\begin{align*}
\norm{\nabla(\varphi \cdot f\circ\phi)}_{L^n(2B_1)}
	& \leq \norm{\nabla\varphi}_{L^\infty} \norm{f\circ\phi}_{L^{n}(2B_1)} + \norm{\varphi}_{L^\infty}\norm{\nabla(f\circ\phi)}_{L^n(2B_1)}.
\end{align*}
 Theorem \ref{theoCompositionQCSobolev} gives
 $$\norm{\nabla( f \circ \phi)}_{L^n(2B_1)}\leq K^\frac1n  \norm{\nabla f}_{L^n(\phi(2B_1))}\leq K^\frac1n  \norm{\nabla f}_{L^n}$$
and using Lemma \ref{lemCompositionQCLebegueRn} and the continuous embedding $W^{1,n}(B_1) \subset L^{\frac{(b+1)n}{b}}(B_1)$ we obtain
 $$\norm{f \circ \phi}_{L^{n}(2B_1)}\leq C_b^{\frac{b}{b+1}\frac1{n}} \norm{f}_{L^{\frac{(b+1)n}{b}}(\phi(2B_1))} \leq C_b^{\frac{b}{b+1}\frac1{n}} \norm{f}_{L^{\frac{(b+1)n}{b}}(B_1)} \leq C_{b,n,C_b}  \norm{f}_{W^{1,n}}.$$
Using again the Poincar\'e inequality we end the proof.

In the supercritical case, we can use the algebra structure of $W^{1,q}(2B_1)$ (see \cite[Section 4.6.4]{RunstSickel}, for instance) to  get 
\begin{align*}
\norm{\varphi \cdot f\circ\phi}_{W^{1,q}}
	& \leq C_{n,q} \norm{\varphi}_{W^{1,q}(2B_1)} \norm{f\circ\phi}_{W^{1,q}(2B_1)}
		=  C_{n,q} \left(\norm{f\circ\phi}_{L^q(2B_1)} + \norm{\nabla(f\circ\phi)}_{L^q(2B_1)}\right).
\end{align*}
By Theorem \ref{theoCompositionQCSobolev} we obtain that 
\begin{align*}
\norm{\varphi \cdot f\circ\phi}_{W^{1,q}}
	& \leq C_{n,q}  \left(\norm{f}_{L^\infty(B_1)} + K^\frac1n C_a^{\frac1n-\frac1p}\norm{\nabla f}_{L^p(\phi(2B_1))}\right)\\
	& \leq C_{K,n,p,q}\left(1+C_a^{\frac1n-\frac1p} \right) \norm{f}_{W^{1,p}}.
\end{align*}

\end{proof}

\begin{remark}\label{remcontinuity}
The proof given above gives constants which blow up when $p\to n$, because they depend on Gagliardo-Niremberg and Morrey inequalities. However, one can show by interpolation with the critical case (see Theorem \ref{theoInterpolationProperty} below) that there exist a constant $C=C_{K,n,p,b,C_b,a,C_a,|B|,|B'|}$ which depends continuously on $p \in \left[\frac{nb+n}{nb+1},\infty\right)$ such that 
$$\norm{T_\phi^B}_{{W}^{1,p} \to {W}^{1,q}} \leq C.$$
\end{remark}

\section{Fractional Spaces}\label{secFractional}
\subsection{A note on interpolation}
When dealing with fractional smoothness there are two families which are studied the most, the Besov and the Triebel-Lizorkin scale, which include a number of classical function spaces, as the H\"older-Zygmund class, the Bessel potentials, the Sobolev traces... (see \cite[Section 2.3.5]{TriebelTheory}). We recall their definition:
\begin{definition}\label{defCollection}Let $\Phi(\R^n)$ be the collection of all the families $\Psi=\{\psi_j\}_{j=0}^\infty\subset C^\infty_c(\R^n)$ such that
\begin{equation*}
\left\{ 
\begin{array}{ll}
\supp \,\psi_0 \subset \DDD(0,2), & \\
\supp \,\psi_j \subset \DDD(0,2^{j+1})\setminus \DDD(0,2^{j-1}) & \mbox{ if $j\geq 1$},\\	
\end{array}
\right.
\end{equation*}
for all multiindex $\alpha\in \N^n$ there exists a constant $c_\alpha$ such that
\begin{equation*}
\norm{D^\alpha \psi_j}_\infty \leq \frac{c_\alpha}{2^{j |\alpha|} } \mbox{\,\,\, for every $j\geq 0$}
\end{equation*}
and
\begin{equation*}
\sum_{j=0}^\infty \psi_j(x)=1 \mbox{\,\,\, for every $x\in\R^n$.}
\end{equation*}
\end{definition}
 
 \begin{definition}
Given any Schwartz function $\psi \in \mathcal{S}(\R^n)$,  its {\em Fourier transform} is
$$F\psi(\zeta)=\int_{\R^n} e^{-2\pi i x\cdot \zeta} \psi(x) dm(x).$$
This notion extends to the tempered distributions $\mathcal{S}(\R^n)'$ by duality (see \cite[Definition 2.3.7]{Grafakos}).
\end{definition}

 \begin{definition}
Let $s \in \R$,  $1\leq p\leq \infty$, $1\leq q\leq\infty$ and $\Psi \in \Phi(\R^n)$. For any tempered distribution $f\in \mathcal{S}'(\R^n)$ we define its {\em non-homogeneous Besov norm}
\begin{equation*}
\norm{f}_{B^s_{p,q}}^\Psi=\norm{\left\{2^{sj}\norm{\left(\psi_j \widehat{f}\right)\widecheck{\,}\, }_{L^p}\right\}}_{l^q},
\end{equation*}
and we call $B^s_{p,q}\subset \mathcal{S}'$ to the set of tempered distributions such that this norm is finite.

Let $s \in \R$,  $1\leq p< \infty$, $1\leq q\leq\infty$ and $\Psi \in \Phi(\R^n)$. For any tempered distribution $f\in \mathcal{S}'(\R^n)$ we define its {\em non-homogeneous Triebel-Lizorkin norm}
\begin{equation*}
\norm{f}_{F^s_{p,q}}^\Psi=\norm{\norm{\left\{2^{sj}\left(\psi_j \widehat{f}\right)\widecheck{\,}\right\}}_{l^q}}_{L^p},
\end{equation*}
and we call $F^s_{p,q}\subset \mathcal{S}'$ to the set of tempered distributions such that this norm is finite. 
\end{definition}
These norms are equivalent for different choices of $\Psi$. Usually one works with radial $\psi_j$ and such that $\psi_{j+1}(x)=\psi_j(x/2)$ for $j\geq 1$. Of course we will omit $\Psi$ in our notation since it plays no role (see \cite[Section 2.3]{TriebelTheory}).

We will work with the so-called Bessel potential spaces 
$$H^{s,p}:=F^s_{p,2}.$$
 For  $s\in \N$ and $1<p<\infty$ we have that $W^{s,p}=H^{s,p}$ and $L^{p}=H^{0,p}$. 


We recall a result from interpolation theory.
\begin{theorem}[see {\cite[Theorem 6.4.5 - (7)]{BerghLofstrom}}]\label{theoLebesgueSobolevInterpolation}
Let $0<s<1$, let $1 < p_0, p_1< \infty$. Let 
$$\frac1{p}=\frac{1-s}{p_0}+\frac{s}{p_1}.$$
 Then
$$(L^{p_0},W^{1,p_1})_{[s]} = H^{s,p},$$
where the interpolation space $(\cdot,\cdot)_{[s]}$ is defined as in \cite[Chapter 4]{BerghLofstrom}.
\end{theorem}
Next, the interpolation property.
\begin{theorem}[See {\cite[Theorem 4.1.2]{BerghLofstrom}}]\label{theoInterpolationProperty}
Let $0<s<1$, let $1 < p_0, p_1, q_0, q_1< \infty$. Let 
$$\frac1{p}=\frac{1-s}{p_0}+\frac{s}{p_1}, \quad\quad \frac1{q}=\frac{1-s}{q_0}+\frac{s}{q_1}.$$
Then, given a linear operator $T: L^{p_0} + W^{1,p_1} \to L^{q_0} + W^{1,q_1}$, it satisfies that 
$$\norm{T}_{H^{s,p}, H^{s,q}}\leq\norm{T}_{L^{p_0}, L^{q_0}}^{1-s}\norm{T}_{W^{1,p_1}, W^{1,q_1}}^s.$$
\end{theorem}

\subsection{Composition in Bessel potential spaces}\label{secBesselPotential}
\begin{proposition}\label{propoCompositionQCSobolevFractionalSub}
 Let $n\ge 2$, $K\geq 1$, $0\leq s \leq 1$ and $1<p<\frac{n}{s}$. Given a  $K$-quasiconformal mapping $\phi:\Omega \to \Omega'$ between two domains in $\R^{n}$, with $2B_1\subset \Omega$ and $\phi(2B_1)\subset B_1$, and real numbers $b$ and $C_b$ satisfying \rf{eqAKBKRn} for $U=2B$, if $\frac1q := \frac1p + \frac{1}{b}\left(\frac1p-\frac{s}{n}\right)<1$,  then there exists a constant $C=C_{K,n,s,p,q}$ such that
\begin{equation}\label{eqTBoundedFractional}
  \norm{T_\phi^{B_1}}_{{H}^{s,p} \to {H}^{s,q}} \leq C C_b^{\frac1p-\frac{s}n}.
 \end{equation} 
\end{proposition}

\begin{proof}
Since $p > q$, we will argue as follows: we will find indices $p_\lambda, q_\lambda$ for $\lambda \in \{0,1\}$, so that 
\begin{equation}\label{eqInterpolateP}
\frac1p=  \frac{1-s}{p_0}+ \frac{s}{p_1}
\end{equation}
and
\begin{equation}\label{eqInterpolateQ}
\frac1q=  \frac{1-s}{q_0}+ \frac{s}{q_1}.
\end{equation}
Moreover, we will need that $(p_0,q_0)$ satisfy the condition in Corollary \ref{coroCompositionLebesgueRn} and $(p_1,q_1)$ satisfy the condition in Corollary \ref{coroCompositionSobolevRn} . Then we will use interpolation. 

We define $p_0,p_1,q_0$ and $q_1$ by the following relations:
\begin{equation}\label{eqDefineSubcriticalPjQj}
\frac{1-\frac1{p}}{n-s}=\frac{1-\frac1{p_\lambda}}{n-\lambda} \mbox{ \quad\quad and \quad\quad }\frac{1-\frac1{q}}{n-s}=\frac{1-\frac1{q_\lambda}}{n-\lambda}
\end{equation}
for $\lambda \in \{0,1\}$ (see Figure \ref{figActionInFractional}).  It follows immediately  that $\frac1{p_\lambda}<\frac1{q_\lambda}<1$. It is a routine to check that \rf{eqInterpolateQ} and \rf{eqInterpolateP} are satisfied. Indeed, using \rf{eqDefineSubcriticalPjQj} we have that
\begin{align*}
1-\frac1p
	& = (1-s)\left(1-\frac1p\right) + s\left(1-\frac1p\right) = (1-s)\left(1-\frac{s}{n}\right)\left(1-\frac1{p_0}\right) + s\left(1+\frac{1-s}{n-1}\right)\left(1-\frac1{p_1}\right).
\end{align*}
Rearranging and using \rf{eqDefineSubcriticalPjQj} again, we get
\begin{align*}
1-\frac1p
	& = (1-s)\left(1-\frac1{p_0}\right) + s \left(1-\frac1{p_1}\right) + s (1-s)  \left(-\frac{1-\frac1{p_0}}{n}  + \frac{1-\frac1{p_1}}{n-1}\right)\\
	& = (1-s)\left(1-\frac1{p_0}\right) + s \left(1-\frac1{p_1}\right),
\end{align*}
showing \rf{eqInterpolateP}. The counterpart \rf{eqInterpolateQ} is proven using the same reasoning. Next we see that 
$$ \frac1{q} - \frac{1}{p} = \frac{1}{b}\left(\frac1{p} - \frac{s}{n}\right)\quad \implies \quad \frac1{q_\lambda} - \frac{1}{p_\lambda} = \frac{1}{b}\left(\frac1{p_\lambda} - \frac{\lambda}{n}\right).$$
So we assume the condition and plugging \rf{eqDefineSubcriticalPjQj} we get
\begin{align*}
\frac1{q_\lambda} - \frac{1}{p_\lambda}
	&  = \left(1-\frac{1}{p_\lambda}\right) - \left(1-\frac1{q_\lambda} \right)
		 = \frac{n-\lambda}{n-s} \left( \frac1{q} -\frac{1}{p}\right)\\
	& = \frac{n-\lambda}{n-s} \frac{1}{b}\left(\frac{1}{p} - \frac{s}{n}\right)
		=  \frac{1}{b}\left(\frac{n-\lambda}{n-s}\left(\frac1{p}-1\right) + \frac{n-\lambda}{n-s} - \frac{(n-\lambda)s}{(n-s)n}\right).
\end{align*}
Again by \rf{eqDefineSubcriticalPjQj} we obtain
\begin{align*}
\frac1{q_\lambda} - \frac{1}{p_\lambda}
	& =  \frac{1}{b}\left(\frac1{p_\lambda} -1 + \frac{n-\lambda}{n-s} - \frac{(n-\lambda)s}{(n-s)n}\right) 
		=  \frac{1}{b}\left(\frac1{p_\lambda} + \frac{ - n^2 + ns + n^2-n\lambda -ns +\lambda s }{(n-s)n}\right)\\
	& =   \frac{1}{b}\left(\frac1{p_\lambda} -\frac{\lambda}{n} \right).
\end{align*}
Thus, we can use again both Corollary \ref{coroCompositionLebesgueRn} to establish 
\begin{equation}\label{eqTBoundedFractionalInLebsegue}
\norm{T_\phi^{B_1}}_{L^{p_0} \to L^{q_0}} \leq   C_b^{\frac1p_0}
\end{equation}
and Corollary \ref{coroCompositionSobolevRn} to find that 
\begin{equation}\label{eqTBoundedFractionalInSobolev}
\norm{T_\phi^{B_1}}_{W^{1,p_1}\to W^{1,q_1}} \leq C_{K,n,p,q}  C_b^{\frac{1}{p_1}-\frac1n}.
\end{equation}
Estimate \rf{eqTBoundedFractional} follows by the interpolation property in Theorem \ref{theoInterpolationProperty}. Namely,
$$\norm{T_\phi^{B_1}}_{H^{s,p} \to H^{s,q} } \leq \norm{T_\phi^{B_1}}_{L^{p_0}\to L^{q_0}}^{1-s} \norm{T_\phi^{B_1}}_{W^{1,p_1}\to W^{1,q_1}}^s \leq  C_{K,n,p_1,q_1}^s  C_b^{\frac1{p_0}(1-s)+\left(\frac1{p_1}-\frac1n\right)s}.$$
\end{proof}

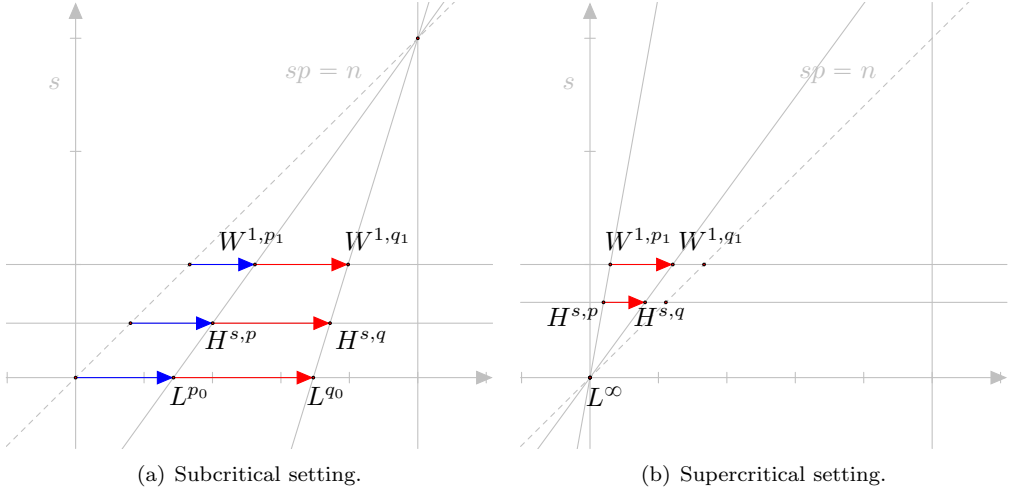
\begin{figure}[ht]
\centering
\subfigure[Subcritical setting.]
{
\begin{tikzpicture}[line cap=round,line join=round,>=triangle 45,x=4.5cm,y=1.5cm]
\draw[->,color=cqcqcq] (-0.20288044871691346,0.) -- (1.219098042976501,0.);
\foreach \x in {-0.2,0.2,0.4,0.6,0.8,1.,1.2}
\draw[shift={(\x,0)},color=cqcqcq] (0pt,2pt) -- (0pt,-2pt);
\draw[->,color=cqcqcq] (0.,-0.6270740036686387) -- (0.,3.3171289076485637);
\foreach \y in {,1.,2.,3.}
\draw[shift={(0,\y)},color=cqcqcq] (2pt,0pt) -- (-2pt,0pt);
\clip(-0.20288044871691346,-0.6270740036686387) rectangle (1.219098042976501,3.3171289076485637);
\draw [color=cqcqcq] (1.,-0.6270740036686387) -- (1.,3.3171289076485637);
\draw [color=cqcqcq,domain=-0.20288044871691346:1.219098042976501] plot(\x,{(-1.-0.*\x)/-1.});
\draw (0.3889485883286697,1.4) node[anchor=north west] {$W^{1,p_1}$};
\draw (0.7583452356121544,1.4) node[anchor=north west] {$W^{1,q_1}$};
\draw [dash pattern=on 2pt off 2pt,color=cqcqcq,domain=-0.20288044871691346:1.219098042976501] plot(\x,{(-0.--3.*\x)/1.});
\draw [color=cqcqcq](0.5835769293705056,2.8285720515639556) node[anchor=north west] {$sp=n$};
\draw [color=cqcqcq](-0.10755228167601415,2.7332438845230564) node[anchor=north west] {$s$};
\draw [->,color=ffqqqq] (0.5239968249699436,1.) -- (0.7963732415936726,1.);
\draw [->,color=qqqqff] (0.3333333333333333,1.) -- (0.5239968249699436,1.);
\draw [color=cqcqcq,domain=-0.20288044871691346:1.219098042976501] plot(\x,{(--0.5719904749098308-2.*\x)/-0.4760031750300564});
\draw [color=cqcqcq,domain=-0.20288044871691346:1.219098042976501] plot(\x,{(--1.3891197247810179-2.*\x)/-0.20362675840632738});
\draw [color=cqcqcq,domain=-0.20288044871691346:1.219098042976501] plot(\x,{(--0.48119770461421985-0.*\x)/1.});
\draw [->,color=qqqqff] (0.16039923487140662,0.4811977046142199) -- (0.40052105506168734,0.48119770461421985);
\draw [->,color=ffqqqq] (0.40052105506168734,0.48119770461421985) -- (0.7435522267620884,0.48119770461421985);
\draw [->,color=qqqqff] (0.,0.) -- (0.2859952374549154,0.);
\draw [->,color=ffqqqq] (0.2859952374549154,0.) -- (0.6945598623905089,0.);
\draw (0.24595633776732073,0.0) node[anchor=north west] {$L^{p_0}$};
\draw (0.6471290407311052,0.0) node[anchor=north west] {$L^{q_0}$};
\draw (0.34922851872829497,0.5) node[anchor=north west] {$H^{s,p}$};
\draw (0.7305411868918921,0.5) node[anchor=north west] {$H^{s,q}$};
\begin{scriptsize}
\draw [fill=ffqqqq] (0.5239968249699436,1.) circle (0.5pt);
\draw [fill=ffqqqq] (0.7963732415936726,1.) circle (0.5pt);
\draw [fill=ffqqqq] (1.,3.) circle (0.5pt);
\draw [fill=ffqqqq] (0.2859952374549154,0.) circle (0.5pt);
\draw [fill=ffqqqq] (0.6945598623905089,0.) circle (0.5pt);
\draw [fill=ffqqqq] (0.40052105506168734,0.48119770461421985) circle (0.5pt);
\draw [fill=ffqqqq] (0.16039923487140662,0.4811977046142199) circle (0.5pt);
\draw [fill=ffqqqq] (0.7435522267620884,0.48119770461421985) circle (0.5pt);
\draw [fill=ffqqqq] (0.,0.) circle (0.5pt);
\draw [fill=ffqqqq] (0.3333333333333333,1.) circle (0.5pt);
\end{scriptsize}
\end{tikzpicture}
}
\subfigure[Supercritical setting.]
{
\begin{tikzpicture}[line cap=round,line join=round,>=triangle 45,x=4.5cm,y=1.5cm]
\draw[->,color=cqcqcq] (-0.20288044871691344,0.) -- (1.219098042976501,0.);
\foreach \x in {-0.2,0.2,0.4,0.6,0.8,1.,1.2}
\draw[shift={(\x,0)},color=cqcqcq] (0pt,2pt) -- (0pt,-2pt);
\draw[->,color=cqcqcq] (0.,-0.6270740036686387) -- (0.,3.3171289076485637);
\foreach \y in {,1.,2.,3.}
\draw[shift={(0,\y)},color=cqcqcq] (2pt,0pt) -- (-2pt,0pt);
\clip(-0.20288044871691344,-0.6270740036686387) rectangle (1.219098042976501,3.3171289076485637);
\draw [color=cqcqcq] (1.,-0.6270740036686387) -- (1.,3.3171289076485637);
\draw [color=cqcqcq,domain=-0.20288044871691344:1.219098042976501] plot(\x,{(-1.-0.*\x)/-1.});
\draw (0.015579934085147445,1.4) node[anchor=north west] {$W^{1,p_1}$};
\draw (0.23006830992717087,1.4) node[anchor=north west] {$W^{1,q_1}$};
\draw [dash pattern=on 2pt off 2pt,color=cqcqcq,domain=-0.20288044871691344:1.219098042976501] plot(\x,{(-0.--3.*\x)/1.});
\draw [color=cqcqcq](0.5835769293705056,2.8285720515639556) node[anchor=north west] {$sp=n$};
\draw [color=cqcqcq](-0.10755228167601415,2.7332438845230564) node[anchor=north west] {$s$};
\draw [->,color=ffqqqq] (0.05927201064555962,1.) -- (0.24197955910407543,1.);
\draw [color=cqcqcq,domain=-0.20288044871691344:1.219098042976501] plot(\x,{(-0.--1.*\x)/0.05927201064555962});
\draw [color=cqcqcq,domain=-0.20288044871691344:1.219098042976501] plot(\x,{(-0.--1.*\x)/0.24197955910407543});
\draw [color=cqcqcq,domain=-0.20288044871691344:1.219098042976501] plot(\x,{(--0.6653072747728902-0.*\x)/1.});
\draw [->,color=ffqqqq] (0.039434099872907004,0.6653072747728902) -- (0.16099076101827794,0.6653072747728902);
\draw (-0.04002816335537714,0.) node[anchor=north west] {$L^{\infty}$};
\draw (-0.15521636519646378,0.7) node[anchor=north west] {$H^{s,p}$};
\draw (0.1,0.7) node[anchor=north west] {$H^{s,q}$};
\begin{scriptsize}
\draw [fill=ffqqqq] (0.05927201064555962,1.) circle (0.5pt);
\draw [fill=ffqqqq] (0.24197955910407543,1.) circle (0.5pt);
\draw [fill=ffqqqq] (0.,0.) circle (0.5pt);
\draw [fill=ffqqqq] (0.,0.) circle (0.5pt);
\draw [fill=ffqqqq] (0.,0.) circle (0.5pt);
\draw [fill=ffqqqq] (0.039434099872907004,0.6653072747728902) circle (0.5pt);
\draw [fill=ffqqqq] (0.2217690915909634,0.6653072747728902) circle (0.5pt);
\draw [fill=ffqqqq] (0.16099076101827794,0.6653072747728902) circle (0.5pt);
\draw [fill=ffqqqq] (0.,0.) circle (0.5pt);
\draw [fill=ffqqqq] (0.3333333333333333,1.) circle (0.5pt);
\end{scriptsize}
\end{tikzpicture}
}
\caption{Interpolation in the proofs  of Propositions \ref{propoCompositionQCSobolevFractionalSub} and  \ref{propoCompositionQCSobolevFractionalSuper}.}\label{figActionInFractional}
\end{figure}
 Proposition \ref{propoCompositionQCSobolevFractionalSub} is proven by an interpolation argument, which is summarized in Figure \ref{figActionInFractional} (a). The idea behind is just the Thales Theorem: horizontal lines correspond to a given smoothness parameter $s$. The homogeneity in the critical line is $0$, and therefore, any pair of straight lines which intersect at the critical line will be suitable for interpolating the results in Lemma \ref{lemCompositionQCLebegueRn} and Theorem \ref{theoCompositionQCSobolev}. When using lines intersecting at $(1,n)$, we make sure that $q_0>1$ and $q_1>1$, which is necessary for the interpolation to make sense.

Next we study the critical case. Here the interpolation limits according to the previous argument should be $L^\infty$ and $W^{1,n}$. But this is not a complex interpolation couple, so we will get uniform bounds by an interpolation of $L^{q_0}$ and $W^{1,n}$ and then we will use Lemma \ref{lemUniformlyBounded} below to end the proof. 
\begin{lemma}\label{lemUniformlyBounded}
Let $0<s<\infty$ and let $1<p_0<\infty$. Let $f\in \mathcal{S}'$ satisfy that for every $\varepsilon$, there exists $p$ with $|p-p_0|<\varepsilon$ such that
$$\norm{f}_{H^{s,p}}\leq 1.$$
Then,
$$\norm{f}_{H^{s,p_0}}\leq C,$$
with C depending on $s$ and $p_0$.
\end{lemma}
\begin{proof}
By the lifting property (see \cite[Theorem 2.3.8]{TriebelTheory}), the linear map $I_s= F^{-1}(1+|x|^2)^{\frac{s}2} F$ induces an isomorphism between $H^{s,r}$ and $L^r$. By interpolation, the norm of this map varies continuously on $1<r<\infty$. Let $\varepsilon>0$ and let $p$ with $|p-p_0|<\varepsilon$ such that
$$\norm{f}_{H^{s,p}}\leq 1.$$
Thus,
$$
\norm{f}_{H^{s,p_0}}\leq C_{s,p_0} \norm{I_s f}_{L^{p_0}} = C_{s,p_0} \left(\norm{I_s f}_{L^{p_0}} -\norm{I_s f}_{L^{p}} \right)+ C_{s,p_0} \norm{I_s f}_{L^{p}} .$$
The last term in the right-hand side above is uniformly bounded by the continuity of $I_s$ and our hypothesis. Namely, for $\varepsilon$ small enough, we can grant
$$\norm{I_s f}_{L^{p}} \leq C_{s,p_0} \norm{f}_{H^{s,p}} \leq  C_{s,p_0}.$$
On the other hand, by the monotone convergence theorem
$$\int |I_s f|^{p} =\int_{|I_s f|\geq 1} |I_s f|^{p} +\int_{|I_s f|<1} |I_s f|^{p}  \xrightarrow{p\to p_0}  \int_{|I_s f|\geq 1} |I_s f|^{p_0} +\int_{|I_s f|<1} |I_s f|^{p_0} = \int |I_s f|^{p_0} .
$$
\end{proof}

\begin{proposition}\label{propoCompositionQCSobolevFractionalCritical}
 Let $n\ge 2$, $K\geq 1$, $0\leq s \leq 1$. Given a  $K$-quasiconformal mapping $\phi:\Omega \to \Omega'$ between two domains in $\R^{n}$, $\phi:\Omega \to \Omega'$ with $2B_1\subset \Omega$ and $\phi(2B_1)\subset B_1$, and real numbers $b$ and $C_b$ satisfying \rf{eqAKBKRn} for $U=2B$,  there exists a constant $C=C_{K,n,s,b,C_b}$ such that
\begin{equation}\label{eqTBoundedFractionalCritical}
  \norm{T_\phi^{B_1}}_{{H}^{s,\frac{n}{s}} \to {H}^{s,\frac{n}{s}}} \leq C_{K,n,s,b,C_b}.
 \end{equation} 
\end{proposition}

\begin{proof}
Let $q$ be such that $0< \frac1{q}-\frac{s}{n}<\varepsilon_0$ with $\varepsilon_0$ small enough.
Let us define $p\in(q,\frac{n}{s})$ by the relation
\begin{equation}\label{eqDefineCriticalAproxP}
\frac1q =: \frac1p +\frac1b\left(\frac1p-\frac{s}{n}\right).
\end{equation}
Let us define $p_0$ and $q_0$ by the following relations:
\begin{equation}\label{eqDefineCriticalPjQj}
\left(\frac1n-\frac1{p}\right)\frac{1}{1-s}=\frac1n-\frac1{p_0} \mbox{ \quad\quad and \quad\quad } \left(\frac1n-\frac1{q}\right)\frac{1}{1-s}=\frac1n-\frac1{q_0} .
\end{equation}
 It follows immediately that $p_0>q_0>n>1$. Let $p_1=q_1=n$. Note that \rf{eqInterpolateP} is equivalent to 
$$\frac1n-\frac1p = (1-s)\left(\frac1n-\frac1{p_0}\right) + s\left(\frac1n-\frac1{p_1}\right),$$
and this holds trivially. The same happens with \rf{eqInterpolateQ}. Let we see that $T^B_\phi:L^{p_0} \to L^{q_0}$. Using identities \rf{eqDefineCriticalAproxP} and \rf{eqDefineCriticalPjQj} we get
\begin{align*}
\frac1n-\frac1{q_0}
	& = \frac1{1-s}\left(\frac1n-\frac1{q}\right)
		= \frac1{1-s}\left(\frac1n - \frac1p -\frac1b\left(\frac1p - \frac1n + \frac1n -\frac{s}{n}\right)\right)\\
	& =  \frac1n - \frac1{p_0}  - \frac1b \left(\frac1{p_0}-\frac1n\right) - \frac1{1-s}\left(\frac1b\left(\frac{1-s}n\right)\right) 
		= \frac1n - \frac1{p_0}  - \frac1b \frac1{p_0},
\end{align*}
so $\frac1{q_0}= \frac1{p_0}  + \frac1b \frac1{p_0}$ with $q_0>1$. By Corollary \ref{coroCompositionLebesgueRn}, we get
$$\norm{T_\phi^{B_1}}_{L^{p_0} \to L^{q_0}} \leq   C_b^{\frac1{p_0}}.$$
By Corollary \ref{coroCompositionSobolevRn} we also have that
$$\norm{T_\phi^{B_1}}_{{W}^{1,p_1} \to {W}^{1,q_1}} = \norm{T_\phi^{B_1}}_{{W}^{1,n} \to {W}^{1,n}} \leq C_{K,n,b,C_b}.$$
Since \rf{eqInterpolateP} and \rf{eqInterpolateQ} hold, using Theorem \ref{theoInterpolationProperty}, 
\begin{equation*}
\norm{T_\phi^{B_1}}_{H^{s,p}\to H^{s,q}}\leq\norm{T_\phi^{B_1}}_{L^{p_0} \to L^{q_0}}^{1-s}\norm{T_\phi^{B_1}}_{W^{1,p_1} \to W^{1,q_1}}^s \leq C_b^{\frac{1-s}{p_0}}    C_{K,n,b,C_b}^s.
\end{equation*} 
Note that $\frac1{p_0}<\frac1{q_0}=\frac1n-\frac1{(1-s)n}+\frac1{q(1-s)}=\frac1{1-s}\left( \frac1{q}-\frac{s}{n}\right)<\frac{\varepsilon_0}{1-s}$. Thus, $C_b^{\frac{1-s}{p_0}} \leq 2$ for $\varepsilon_0<\frac{\log(2)}{\log(C_b)}$.

Take $\varphi_{B_1}$ and let  $\widetilde{f}=\varphi_{B_1} f$. Then $T_\phi^{B_1} f=T_\phi^{B_1} \widetilde{f}$. Thus, writing $c_0=2C_{K,n,b,C_b}^s$, we get
\begin{align*}
\norm{T_\phi^{B_1} f}_{H^{s,q}}
	& =\norm{T_\phi^{B_1} \widetilde{f}}_{H^{s,q}} 
		\leq c_0 \norm{\widetilde{f}}_{H^{s,p}}
		\leq c_0 C_{n,s} \norm{\widetilde{f}}_{H^{s,\frac ns}}\\
	& \leq c_0 C_{n,s} \norm{\varphi_{B_1} \cdot}_{H^{s,\frac{n}{s}}\to H^{s,\frac{n}{s}}}\norm{f}_{H^{s,\frac{n}{s}}},
\end{align*}
where $\varphi_{B_1} \cdot : f\mapsto \varphi_{B_1} f$ stands for the pointwise multiplication operator, which has norm one in every $L^p$. Thus, by the interpolation property, it can be uniformly bounded by a constant $C$, so
$$  \norm{T_\phi^{B_1}}_{{H}^{s,\frac{n}s} \to {H}^{s,q}} \leq  C_{K,n,b,C_b}^s C_{n,s}.
$$
Lemma \ref{lemUniformlyBounded} implies that this uniform bound applies to the limit case, modulo constants depending on $n$ and $s$.
\end{proof}

The supercritical case follows the same pattern. The interpolation limits according to the previous argument should be $L^\infty$ for the $0$-indices and $W^{1,sp}$ or $W^{1,sq}$ for the $1$-indices (see Figure \ref{figActionInFractional} (b)). Again, this is not a complex interpolation couple, so we will follow the approximation procedure above. Moreover, in this context we need to use the parameter $a$ for the classical Sobolev spaces, while we are forced to use the parameter $b$ in the Lebesgue spaces. When taking limits, however, the parameter $b$ will vanish.

\begin{proposition}\label{propoCompositionQCSobolevFractionalSuper}
 Let $n\ge 2$, $K\geq 1$, $0\leq s \leq 1$ and $\frac{n}{s}<p<\infty$. Given a  $K$-quasiconformal mapping $\phi:\Omega \to \Omega'$ between two domains in $\R^{n}$, with $2B_1\subset \Omega$ and $\phi(2B_1)\subset B_1$, and positive real numbers $a$, $b$, $C_a$ and $C_b$ satisfying \rf{eqAKBKRn} for $U=2B$, if $\frac1q := \frac1p + \frac{1}{a}\left(\frac{s}{n}-\frac1p\right) < 1$,  then there exists a constant $C=C_{K,n,s,p,a}$ such that
\begin{equation}\label{eqTBoundedFractionalSupercritical}
  \norm{T_\phi^{B_1}}_{{H}^{s,p} \to {H}^{s,q}} \leq C \left(1+C_a^{\frac{s}{n}-\frac{1}{p}}\right).
 \end{equation} 
\end{proposition}
\begin{proof}
Let $0<\frac1{q_0}<\varepsilon_0<\frac1p$ with $\varepsilon_0$ small enough.
Let us define $p_0$ by the relation
$$\frac1{q_0}= \frac1{p_0}  + \frac1b \frac1{p_0}.$$
Note that $\frac1{p_0}<\frac1{q_0}<\frac1{p}$. Next, we define $p_1$ by imposing \rf{eqInterpolateP}, that is, $\frac1p=  \frac{1-s}{p_0}+ \frac{s}{p_1}$. In particular, $\frac1{p_1}=\frac1{sp}-\frac{1-s}{sp_0}<\frac1n$. Thus, we are in the supercritical range and $\left|\frac1{p_1}-\frac1{sp}\right|<\frac{(1-s)\varepsilon_0}s$. Let us define $q_1$ such that $T_\phi^{B_1}$ maps $W^{1,p_1}$ to $W^{1,q_1}$, that is,
$$\frac1{q_1} := \frac1{p_1} + \frac{1}{a}\left(\frac{1}{n}-\frac1{p_1}\right).$$
Using the definitions we get
$$\left|\frac1{q_1}-\frac1{sq}\right|=\left|\frac{1}{p_1}\left(1-\frac{1}{a}\right)+\frac{1}{an}-\frac{1}{sp}\left(1-\frac{1}{a}\right)-\frac{s}{san}\right|<\frac{(1-s)\varepsilon_0}s\frac{a-1}{a}.$$
 Finally, we define $\widetilde{q}$ via \rf{eqInterpolateQ}, that is 
\begin{equation}\label{eqInterpolateQTilde}
\frac1{\widetilde{q}}=  \frac{1-s}{q_0}+ \frac{s}{q_1}.
\end{equation}
Note that 
$$\left|\frac1{\widetilde{q}}-\frac1q\right|=\left|\frac{1-s}{q_0} + \frac s{q_1}-\frac1q\right| \leq  \frac{1-s}{q_0} + s\left|\frac1{q_1}-\frac1{sq}\right|\leq \varepsilon_0 (1-s) \frac{2a-1}{a}.$$

From Corollary \ref{coroCompositionLebesgueRn} we have that
 $$\norm{T_\phi^{B_1}}_{L^{p_0} \to L^{q_0}} \leq   C_{b}^{\frac1{p_0}},$$
and from Corollary \ref{coroCompositionSobolevRn} we have that
$$\norm{T_\phi^{B_1}}_{{W}^{1,p_1} \to {W}^{1,q_1}} \leq C_{K,n,p_1,q_1}\left(1+C_a^{\frac1n-\frac1{p_1}} \right).$$
Using Theorem \ref{theoInterpolationProperty}, 
\begin{equation*}
\norm{T_\phi^{B_1}}_{H^{s,p}\to H^{s,\widetilde{q}}}
	\leq\norm{T_\phi^{B_1}}_{L^{p_0} \to L^{q_0}}^{1-s}\norm{T_\phi^{B_1}}_{H^{s,p_1} \to H^{s,q_1}}^s 
	\leq C_{b}^{\frac{1-s}{p_0}} C_{K,n,p_1,a}^s\left(1+C_a^{\frac sn-\frac s{p_1}} \right).
\end{equation*} 
On one hand, we have that $C_{b}^{\frac{1-s}{p_0}} \leq 2$ for $\varepsilon_0<\log(2)/\log(C_{b}^{1-s})$. On the other hand, from Remark \ref{remcontinuity}, the constant $C_{K,n,p_1,a}^s$ appearing in Corollary \ref{coroCompositionSobolevRn} is uniformly bounded in $p_1$ for $\left|\frac1{p_1}-\frac1{sp}\right|<\frac{(1-s)\varepsilon_0}s<\frac12 \left|\frac1{sp}-\frac1n\right|$.  Thus
\begin{equation*}
\norm{T_\phi^{B_1}}_{H^{s,p}\to H^{s,\widetilde{q}}}
	\leq  C_{K,n,s,p,a}\left(1+C_a^{\frac sn-\frac1p} \right).
\end{equation*} 
Lemma \ref{lemUniformlyBounded} implies that this uniform bound applies to the limit case, modulo constants depending on $n$ and $s$.
\end{proof}

\begin{proof}[Proof of \rf{eqTBoundedFractionalGeneral}]
Propositions \ref{propoCompositionQCSobolevFractionalSub}, \ref{propoCompositionQCSobolevFractionalCritical} and \ref{propoCompositionQCSobolevFractionalSuper} show that \rf{eqTBoundedFractionalGeneral} holds for $B=B'=B_1$. Estimate \rf{eqTBoundedFractionalGeneral} follows by translation and rescaling. 
\end{proof}

\subsection{Other fractional spaces}\label{secOtherSpaces}
\begin{theorem}[see {\cite[Theorems 2.4.2, 2.4.3, 2.4.7 and remark 2.4.7/2]{TriebelTheory}}]
Let $-\infty <s_0,s_1<\infty$, let $0 < r \leq \infty$, $0<p<\infty$ and $0<\Theta <1$. Let 
$$s=(1-\Theta)s_0+\Theta s_1.$$

 If  $s_0\neq s_1$, then
\begin{equation}\label{eqInterpolationRealTriebelLizorkin}
(H^{s_0,p},H^{s_1,p})_{\Theta,r} = B^s_{p,r},
\end{equation}
where $(\cdot,\cdot)_{\Theta,r}$ stands for the real interpolation functor described in \cite{BerghLofstrom}.
\end{theorem}

The functor $(\cdot,\cdot)_{\Theta,r}$ is an exact interpolation functor of exponent $\Theta$, i.e. it satisfies the interpolation property described in Theorem \ref{theoInterpolationProperty}. Thus, using \rf{eqInterpolationRealTriebelLizorkin}, estimate  \rf{eqTBoundedFractionalGeneral} on neighboring $H^{\widetilde{s},p}$ spaces, and a limiting argument  with a Besov version of Lemma \ref{lemUniformlyBounded} we can argue as in the previous section, to get
$$\norm{T_\phi^B}_{B^s_{p,r} \to B^s_{q,r}} \leq C$$
for any $0<s<1$ and $0<r\leq\infty$.

\section{Sharpness}\label{secExamples}
In this section we provide examples that show the sharpness of Theorem \ref{theoCompositionQCSobolevFractionalMain}. 

For $k>0$, let $\phi_{k}(x):=x|x|^{k-1}$. We have that
\[D\phi(x)=|x|^{k-1}\left((\alpha-1)\frac{xx^{t}}{|x|^{2}}+Id\right).\]
Hence, we obtain $J_{\phi_{k}}(x)=k|x|^{n(k-1)}$, and, therefore, $\phi_{k}$ is $k^{n-1}$-quasiconformal if $k\ge 1$ and $\frac{(2-k)^{n}}{k}$-quasiconformal if $k<1$. It also satisfies \eqref{eqAKBKRn} for $a<\frac{1}{1-k}$ and $b>0$ if $k<1$ and for $a>1$ and $0<b<\frac{1}{k-1}$ if $k>1$ regardless of the chosen domain $U$. 

Next we recover the example  given in \cite{HenclKoskelaComposition} to show the sharpness on the subcritical setting (that is, when $sp<n$) for the composition of an unbounded $B^s_{p,p}$ function function with a quasiconformal mapping. Let
\[f_{\rho}(x):=\max\{|x|^{-\rho}-1,0\}\text{ with }\rho>0.\]
It is known that $f_{\rho}$ belongs to the space $H^{s,p}$ if and only if $0<\rho<\frac{n}{p}-s$ (see \cite[Lemma 2.3.1]{RunstSickel}). Let $b>0$ and $q\ge 1$ such that
\(\frac1q < \frac1p + \frac{1}{b}\left(\frac1p-\frac{s}{n}\right),\)
let 
\begin{equation}\label{eqVEDefinition}
\ve:=\frac1p + \frac{1}{b}\left(\frac1p-\frac{s}{n}\right)-\frac1q >0
\end{equation} 
and $\delta>0$ such that
\[(1-\delta)^{2}\left(\frac{n}{q}-s+n\ve\right)>\frac{n}{q}-s.\]
Define 
\begin{equation}\label{eqDelta}
k:=(1-\delta)\left(\frac{1}{b}+1\right)
\end{equation}
 and $\rho$ such that
 \begin{equation}\label{eqRho}
 (1-\delta)\left(\frac{n}{p}-s\right)<\rho<\frac{n}{p}-s.
 \end{equation}
Then, $\phi_{k}$ satisfies \eqref{eqAKBKRn} for $b$, and $f_{\rho}\in H^{s,p}$. On the other hand we have $f_{\rho}\circ\phi_{k}= f_{k\rho}$ and combining \rf{eqDelta} and \rf{eqRho} with \rf{eqVEDefinition}, we get 
\begin{align*}
&k\rho-\left(\frac{n}{q}-s\right)>(1-\delta)^{2}\left(\frac{1}{b}+1\right)\left(\frac{n}{p}-s\right)-\left(\frac{n}{q}-s\right)\\
&=(1-\delta)^{2}\left(\frac{n}{q}-s+n\ve\right)-\left(\frac{n}{q}-s\right)>0.
\end{align*}
Therefore, $f_{\rho}\circ\phi_{k}\notin H^{s,q}$. The interested reader may find in \cite[Lemma 4.2]{HenclKoskelaComposition} an example of a bounded function satisfying the same.

For the supercritical case, i.e. $sp>n$, we define
\[g_{\rho}(x):=\max\{1-|x|^{\rho},0\}\text{ with }\rho>0.\]
 Then, using again \cite[Lemma 2.3.1]{RunstSickel} we have that $g_{\rho}$ belongs to the Triebel-Lizorkin space $H^{s,p}$ if and only if $s-\frac{n}{p}<\rho$. Let $a> 1$ and $q\ge 1$ such that \(\frac1q <\frac1p + \frac1{a}\left(\frac{s}{n}-\frac1p\right),\) let 
 $$\ve:=\frac1p + \frac{1}{a}\left(\frac{s}{n}-\frac1p\right)-\frac1q$$
  and $\delta>0$ such that
\[(1+\delta)^{2}\left(s-\frac{n}{q}-n\ve\right)<s-\frac{n}{q}.\]
Define 
$$k:=(1+\delta)\left(1-\frac{1}{a}\right)$$
 with sufficient small $\delta$ to have  $k<1$, and $\rho$ such that 
 $$s-\frac{n}{p}<\rho<(1+\delta)\left(s-\frac{n}{p}\right).$$
  Then, $\phi_{k}$ satisfies \eqref{eqAKBKRn} for $a$, and $g_{\rho}\in H^{s,p}$. On the other hand we have $g_{\rho}\circ\phi_{k}= g_{k\rho}$ and arguing as before we obtain
\begin{align*}
&s-\frac{n}{q}-k\rho>s-\frac{n}{q}-(1+\delta)^{2}\left(1-\frac{1}{a}\right)\left(s-\frac{n}{p}\right)\\
&=s-\frac{n}{q}-(1+\delta)^{2}\left(s-\frac{n}{q}-n\ve\right)>0.
\end{align*}
Therefore, $f_{\rho}\circ\phi_{k}\notin H^{s,q}$.

\renewcommand{\abstractname}{Acknowledgements}
\begin{abstract}
The authors were funded by the European Research
Council under the grant agreement 307179-GFTIPFD. The authors want to thank Stanislav Hencl, Carlos Mora-Corral and Daniel Faraco for their contributions.

\end{abstract}

\bibliography{../../../bibtex/Llibres}
\end{document}